\documentclass[12pt]{amsart}
\usepackage[utf8]{inputenc}
\usepackage{graphicx}
\usepackage{epstopdf}
\usepackage{inputenc}
\usepackage{enumitem}
\usepackage{amsmath}
\usepackage{xcolor}
\usepackage{amssymb}
\usepackage{mathrsfs}

\usepackage[utf8]{inputenc}
\usepackage{graphicx}
\usepackage{epstopdf}
\usepackage{inputenc}
\usepackage{enumitem}
\usepackage{amsmath}
\usepackage{xcolor}
\usepackage{amssymb}
\usepackage{dsfont}
\usepackage[colorlinks]{hyperref}

\def\house#1{\setbox1=\hbox{$\,#1\,$}%
	\dimen1=\ht1 \advance\dimen1 by 2pt \dimen2=\dp1 \advance\dimen2 by 2pt
	\setbox1=\hbox{\vrule height\dimen1 depth\dimen2\box1\vrule}%
	\setbox1=\vbox{\hrule\box1}%
	\advance\dimen1 by .4pt \ht1=\dimen1
	\advance\dimen2 by .4pt \dp1=\dimen2 \box1\relax}

\newtheorem{theorem}{Theorem}
\newtheorem*{theorem*}{Theorem}

\newtheorem{lemma}{Lemma}

\newtheorem*{acknowledgements*}{Acknowledgements}

\newcommand{\ve}{\varepsilon}

\newcommand\blfootnote[1]{
	\begingroup
	\renewcommand\thefootnote{}\footnote{#1}
	\endgroup
}

\begin{document}
\title[On the Distribution and Maximal Behavior]{On the Distribution and Maximal Behavior of $L(1, \chi_D)$ over Hyperelliptic Curves}
\author{Pranendu Darbar}
\address[Pranendu Darbar]{School of Mathematics and Statistics, \\ University of New South
	Wales, Sydney NSW 2052, Australia}
\email{darbarpranendu100@gmail.com}


 \begin{abstract}
 We improve the range of uniformity in the double-exponential decay of the tail of the distribution established by Lumley~\cite{Lumley} for the quadratic Dirichlet $L$-function $L(1, \chi_D)$ over the ensemble of hyperelliptic curves of genus~$g$ defined over a fixed finite field~$\mathbb{F}_q$, in the limit as $g \to \infty$. Furthermore, we apply a long resonator method to show that this range of uniformity may persist up to its conjectural level by establishing a double-exponential decay lower bound for the corresponding distribution function.
   \end{abstract}


\maketitle
  
\section{Introduction} 
One of the central problems in analytic number theory is to determine the value distribution and obtain sharp bounds for the maximal size of Dirichlet $L$-functions on the line $\operatorname{Re}(s)=1$. These quantities are deeply intertwined with fundamental arithmetic invariants such as class numbers of number fields, and are closely linked to estimates for character sums via Fourier analysis. In particular, the tail behavior of the value distribution plays a decisive role in determining the true order of magnitude of extreme values.

Over the past few decades, significant progress has been made toward understanding the tails of these distributions and the maximal size of $L$-functions on $\operatorname{Re}(s)=1$. This line of research encompasses the Riemann zeta function and Dirichlet $L$-functions in the modulus aspect (see~\cite{AMM, AMMP, BS3, DOKIC, GSzeta}), as well as the family of quadratic Dirichlet $L$-functions $L(1,\chi_d)$ (see~\cite{DM, GS, Lumley}), and more generally, various families of automorphic $L$-functions (see~\cite{Lamzouri}).

An important $\mathrm{GL}(1)$ family of symplectic type is formed by the quadratic Dirichlet $L$-functions $L(s,\chi_d)$, where $d$ ranges over fundamental discriminants. The analytic study of this family is notably challenging due to the lack of a natural orthogonality relation among the characters. The tail behavior of $L(1,\chi_d)$ is closely tied to the distribution of class numbers of quadratic number fields. In this direction, Granville and Soundararajan~\cite{GS} carried out a seminal analysis of these quantities, and following their approach, Dahl and Lamzouri~\cite{DL} investigated an interesting subfamily, known as the \emph{Chowla family}, corresponding to real quadratic fields.


\subsection{Distribution of $L$-functions in the hyperelliptic ensemble}\label{large values section}
We now move our discussion to the setup of $L$-functions over function fields, a topic of interest in modern research in number theory; see, e.g., \cite{AT, AK, BF, BF2, DFL, DL, FR, Floreanegative, Florea, Jung, Lumley2}.
This includes investigations of various statistical properties of these $L$-functions, such as moments, non-vanishing, and value distribution.
 
 Before we enunciate the main theorem of this paper, we need to present
 some basic notations on function fields.
  Let $\mathbb{F}_{q}$ be a finite field of odd cardinality and $\mathbb{F}_{q}[t]$ be the polynomial ring over $\mathbb{F}_{q}$ in variable $t$. Let $D\in \mathbb{F}_q[t]$ be a monic square-free polynomial. The quadratic character $\chi_D$ attached  to $D$ is defined using quadratic residue symbol for $\mathbb{F}_{q}[t]$ by $\chi_{D}(f)=\left(\frac{f}{D}\right)$ and the corresponding Dirichlet $L$-function is denoted by $L(s, \chi_D)$. 
 Let $\mathcal{H}_n$
 be the family of all curves given in affine form by $C_P: y^2=D(t)$, where $D$ be a monic square-free polynomial of degree $n$. These curve are non-singular of genus $g$ given by \eqref{lambda} and \eqref{genus}.

 In~\cite{Lumley}, Lumley studied the distribution of $L(1,\chi_D)$ as $D$ ranges over $\mathcal{H}_n$.  
 The following theorem strengthens the range of uniformity for the tail of the distribution of $L(1,\chi_D)$ established in~\cite[Theorem~1.2]{Lumley}.
Throughout, $\log$ denotes the logarithm to base $q$, while $\ln$ is the natural logarithm, and $\log_j$ (respectively $\ln_j$) for the $j$-fold iterated logarithm.
 \begin{theorem}\label{main theorem}
 	Let $q \ge 3$ be fixed. Uniformly for 
 $
 	\tau \le \log n + \log_2 n - \theta(n),
 	$
 	where $2\leq \theta(n)\ll \log_3 n$ is any function tending arbitrarily slowly to infinity as $n \to \infty$, we have
 	\[
 	\phi_n(\tau)
 	:= \frac{1}{|\mathcal{H}_n|}
 	\sum_{\substack{D \in \mathcal{H}_n \\ L(1, \chi_D) \ge e^{\gamma}\tau}} 1
 	= \exp\!\left(
 	-\,C_1\!\left(q^{\{\log \kappa(\tau)\}}\right)
 	\,\frac{q^{\,\tau - C_0\!\left(q^{\{\log \kappa(\tau)\}}\right)}}{\tau}
 	\bigl(1 + o_\theta(1)\bigr)
 	\right),
 	\]
 	where the constants $C_0\!\left(q^{\{\log \kappa(\tau)\}}\right)$ and $C_1\!\left(q^{\{\log \kappa(\tau)\}}\right)$ depend on $\tau$, $q$, and $\kappa(\tau)$, yet remain bounded as the argument varies between $1$ and $q$.  
 	Further details on these quantities can be found in~\cite[Theorem~1.3]{Lumley}.
 \end{theorem}
 
 We emphasize that Theorem~\ref{main theorem} subsequently improves the uniformity range of~\cite[Theorem~1.2]{Lumley} from 
 $\log n - 2\log_2 n - \log_3 n$ to $\log n + \log_2 n - \theta(n)$, where $\theta(n)$ grows arbitrarily slowly to infinity as $n \to \infty$. A similar range of uniformity as in Lumley’s work is also found in the work of Dahl and Lamzouri~\cite[Theorem~1.2]{DL}. 
 For a discussion on the significance of these uniformity ranges, see~\cite{MV}.   
Motivated by the work of Granville and Soundararajan~\cite{GS}, Lumley conjectured that as $n\to \infty$,
 \[
 \max_{D\in \mathcal{H}_{n}} L(1,\chi_D) = \log n + \log_2 n + O(1).
 \]
 It is plausible that if the asymptotic formula in Theorem~\ref{main theorem} were to hold up to the boundary of the admissible range, then the above conjecture would follow.  
 Our result, however, falls slightly short of this limit, requiring that $\theta(n)$ remain bounded as stated in Theorem~\ref{main theorem}.  
 Nevertheless, this constitutes the first unconditional result approaching the conjectural range, comparing upon the works of~\cite{DDLL, DL, GS, GSzeta, Lamzouri, Lumley}.  
 Conditionally on the Generalized Riemann Hypothesis (GRH), the corresponding result was established by Granville and Soundararajan~\cite[Theorem~4]{GS} in the number field case.  
 
Finally, by employing a long resonator method, we provide further evidence that the double-exponential decay may persists up to the conjectured level, where $\theta(n)$ transitions from an arbitrarily slowly unbounded function to a bounded one.

\begin{theorem}\label{theorem on resonance method}
	Let $q \ge 3$ and for every $\beta > 0$, define $\tau_{\beta, n}
	= e^{\gamma}\!\left(\log n + \log_2 n + C_2(q) - \beta\right),$ where
	\[
	C_2(q)
	= \frac{1}{2}
	- \left(\frac{\pi}{4} - \frac{\ln 2}{2}\right)\frac{q}{q-1}
	+ \log\!\left(
	\frac{(q-1)\ln q}{2q(3\ln 2 - \pi/2)}
	\right).
	\]
	Then we have
	\[
\phi_n(\beta)
:= \frac{1}{|\mathcal{H}_n|}
\sum_{\substack{D \in \mathcal{H}_n \\ L(1, \chi_D) \ge \tau_{\beta, n}}} 1
	\ge \exp\!\left(
	-\frac{q^{-\beta}\ln q}{2}
	\left(1 + o(1)\right)
	\right).
	\]
\end{theorem}
The best previously known unconditional result in this direction over number fields was due to Granville and Soundararajan~\cite[Theorem~5b]{GS}, while the conditional analogue under the GRH was obtained by the author and Maiti~\cite[Theorem~2]{DM}.

During the proof of Theorem~\ref{theorem on resonance method}, we also observed that
\[
\max_{D\in \mathcal{H}_n} L(1,\chi_D)
\ge e^{\gamma}\Bigl(\log n + \log_2 n + C_2(q) + o(1)\Bigr),
\]
where \(C_2(q)\) is as defined in the statement of the theorem. This result refines~\cite[Theorem~1.6]{Lumley} by providing an explicit value for the constant \(C_2(q)\). More strikingly, one has \(C_2(q) > 0\) for \(q > 10\); for instance, \(C_2(17) \approx 0.04\). This is the first time a positive constant is observed, showing that the maximum of \(L(1,\chi_D)\) can exceed the range \(\log n + \log_2 n\) (see~\cite[Conjecture~2]{MV}).
 
\subsubsection{Applications} For a monic square-free polynomial $D \in \mathbb{F}_q[t]$, define 
\[
h_D = \lvert \mathrm{Pic}(\mathcal{O}_D) \rvert,
\]
where $\mathrm{Pic}(\mathcal{O}_D)$ denotes the Picard group of the ring of integers $\mathcal{O}_D \subset \mathbb{F}_q[t](\sqrt{D(t)})$. 
Artin~\cite{Artin} established a class number formula over the hyperelliptic ensemble that connects $h_D$ and $L(1, \chi_D)$:
\begin{equation*}
L(1, \chi_D) = \frac{\sqrt{q}}{\sqrt{|D|}}\, h_D = q^{-g} h_D 
\quad \text{for } D \in \mathcal{H}_{2g+1}.
\end{equation*}
From Theorem~\ref{main theorem}, we directly extend the range of uniformity in \cite[Corollary~1.8]{Lumley}, showing that the tail of the distribution of large values of $h_D$ over $\mathcal{H}_{2g+1}$ decays doubly exponentially. An analysis towards maximal value for $h_D$ follows from Theorem~\ref{theorem on resonance method}.

For $n = 2g + 2$ and $D \in \mathcal{H}_{2g+2}$, Artin also proved that
\begin{equation*}
L(1, \chi_D) = \frac{q - 1}{\sqrt{|D|}}\, h_D R_D,
\end{equation*}
where $R_D$ denotes the regulator of $\mathcal{O}_D$ (see \cite[Chapter 14]{ROS}). 
Analogous results hold for the tail of the distribution of $h_D R_D$ as $D$ varies over $\mathcal{H}_{2g+2}$, improving~\cite[Corollary~1.9]{Lumley} in a similar manner.
 \subsection{Essence of the paper} We employ two distinct methods to prove Theorems~\ref{main theorem} and~\ref{theorem on resonance method}. 
 The study of the distribution of $L(1, \chi_d)$ over number fields through an underlying probabilistic model was initiated by Granville and Soundararajan~\cite{GS}. 
 Lumley~\cite{Lumley} studied the tail of the distribution for $L(1, \chi_D)$ over function fields from the corresponding probabilistic model (see \cite[eq.~(1.7)]{Lumley}) and computing large moments og $L(1, \chi_D)$ over the family. 
 To extend the range of uniformity stated in Theorem~\ref{main theorem}, we able to compute much higher moments for the short Euler product $L(1, \chi_D; N)$ and then utilize the saddle-point method. 
 The Lemma~\ref{full l to its truncation} provides an optimal way to approximate $L(1, \chi_D)$ by $L(1, \chi_D; N)$, allowing us to taken a shorter truncation length $N$. 
 This step relies crucially on a large sieve type estimate together with Chebyshev’s inequality. 
 
 To prove Theorem~\ref{theorem on resonance method}, we employ the long resonance method for quadratic Dirichlet $L$-functions over number fields, as developed by the author and Maiti in~\cite{DM}. 
 The character sum estimate in Lemma~\ref{character sum over nonsquare} plays a central role in both results: it extends the range of uniformity as in Theorem \ref{main theorem} via computing larger moments of $L(1, \chi_D; N)$ over the family, and it enables the control of long resonators in Theorem~\ref{theorem on resonance method}. These ideas appear difficult to implement for higher degree $L$-functions but may still be applicable to families of $\mathrm{GL}(1)$ $L$-functions associated with higher order characters.

 \section{Preliminaries and Lemmas} 
\subsection{Setup}

We begin by fixing a finite field $\mathbb{F}_{q}$ of odd cardinality $q = p^{r}$, where $p$ is an odd prime and $r \ge 1$. For simplicity, we assume that $q \equiv 1 \pmod{4}$.  
Let $\mathbb{A} := \mathbb{F}_{q}[t]$ denote the polynomial ring over $\mathbb{F}_{q}$.

For a polynomial $f \in \mathbb{F}_{q}[t]$, we write $\deg(f)$ or $d(f)$ for its degree, and define its norm by $|f| := q^{d(f)}$ (with $|0| := 0$). 

\vspace{1mm}
\noindent
Let $\mathcal{M}_{n,q}$ (or simply $\mathcal{M}_n$) denote the set of monic polynomials of degree $n$, and $\mathcal{P}_{n,q}$ (or simply $\mathcal{P}_n$) the set of monic irreducible polynomials of degree $n$.  
We put
\[
\mathcal{M} := \bigcup_{n \ge 1} \mathcal{M}_n, \quad 
\mathcal{P} := \bigcup_{n \ge 1} \mathcal{P}_n, \quad 
\mathcal{M}_{\le n} := \bigcup_{m \le n} \mathcal{M}_m.
\]

For any real number $r \ge 1$, the generalized divisor function $d_r(f)$ is defined on prime powers by
\begin{equation}\label{generalized_divisor}
d_r(P^a) := \frac{\Gamma(r + a)}{\Gamma(r)\, a!},
\end{equation}
and extended to all monic polynomials multiplicatively.

\vspace{1mm}
\noindent
The Prime Polynomial Theorem (see \cite[Theorem~2.2]{ROS}) asserts that
\begin{equation}\label{prime poly th}
\pi_q(n) := |\mathcal{P}_n|
= \frac{q^n}{n} + O\!\left(\frac{q^{n/2}}{n}\right).
\end{equation}
Finally, the zeta function of $\mathbb{A}$ is given by
\[
\zeta_{\mathbb{A}}(s)
:= \sum_{f \in \mathcal{M}} \frac{1}{|f|^{s}}
= \prod_{P \in \mathcal{P}} \left(1 - |P|^{-s}\right)^{-1},
\qquad \operatorname{Re}(s)>1.
\]

\subsection{Quadratic Dirichlet character and properties of their $L$-functions} For a monic irreducible polynomial $P$, the quadratic residue symbol $\left(\frac{f}{P} \right)$ is defined by 
\begin{align*}
\left( \frac{f}{P}\right)= 
\left\{
\begin{array}
[c]{ll}
1 & \text{if\, $f$ is a square\;} (\text{mod}\,\,  P),\,\, P\nmid f \\
-1 & \text{if\, $f$ is not a square\;} (\text{mod}\,\, P),\,\, P\nmid f\\
0 & \text{if \; } P\mid f.
\end{array}
\right.
\end{align*}
The quadratic Dirichlet character $\chi_{D}$ is defined by $\chi_{D}(f)= \left(\frac{f}{D} \right)$ and the associated $L$-function by
\begin{align*}
L(s,\chi_{D})=\sum_{f\in \mathcal{M}} \frac{\chi_{D}(f)}{|f|^{s}}=\prod_{P\in \mathcal{P}}\left(1-\chi_{D}(P)\,|P|^{-s} \right)^{-1},\;\; \operatorname{Re}(s)>1.
\end{align*} 
Using the change of variable $u=q^{-s}$, we have  
\begin{align*}
\mathcal{L}(u,\chi_{D})=\sum_{f\in \mathcal{M} } \chi_{D}(f)\, u^{d(f)}=\prod_{P\in\mathcal{P}}\left(1-\chi_{D}(P)\,u^{d(P)} \right)^{-1},\;\;\, |u|<1/q .
\end{align*}
\noindent
By [\cite{ROS}, Proposition $4.3$], we see that if $n\geq d(D)$ then $$\sum_{f\in \mathcal{M}_{n}}\chi_{D}(f)=0 .$$ It implies that $\mathcal{L}(u,\chi_{P})$ is a polynomial of degree at most $d(D)-1$. From \cite{Rud}, $\mathcal{L}(u, \chi_D)$ has a trivial zero at $u=1$ if and only if $d(D)$ is even. This allows us to define the completed $L$-function as
\[
L(s,\chi_{D})=\mathcal{L}(u, \chi_D)=(1-u)^{\lambda}\mathcal{L}^{*}(u, \chi_P)=(1-q^{-s})^{\lambda}{L}^{*}(s, \chi_D),
\]
where 
\begin{align}\label{lambda}
\lambda= 
\left\{
\begin{array}
[c]{ll}
1, & \;\text{if\, $d(D)$ even}, \\
0, &\; \text{if\, $d(D)$ odd},
\end{array}
\right.
\end{align}
and $\mathcal{L}^*(u, \chi_D)$ is a polynomial of degree 
\begin{align}\label{genus}
2g=d(D)-1-\lambda
\end{align}
satisfying the functional equation 
\[
\mathcal{L}^*(u, \chi_D)=(qu^2)^{g}\mathcal{L}^*\left(\frac{1}{qu}, \chi_D\right).
\] 
The Riemann hypothesis for curve over finite fields, established by Weil \cite{WEIL}, asserts that all the non-trivial zero of $\mathcal{L}^{*}(u, \chi_{D})$ are lie on the circle $|u|=q^{-1/2}$. Letting $e(x) := e^{2 \pi i x}$, we then write
\begin{align*}\label{20250908_17:40}
\mathcal{L}^*(u,\chi_{D})= \prod_{j=1}^{2g} \left(1 -  \frac{u}{u_j}\right)\,,
\end{align*}
where $u_j := q^{-1/2} e(\theta_j)$ are the zeros of $\mathcal{L}^*(u,\chi_{D})$.
\subsection{Orthogonality of characters over the family}
The following estimate describes the mean value of characters over the hyperelliptic ensemble when evaluated through square polynomials, see \cite[Lemma 3.7]{BF}. 
\begin{lemma}\label{character sum over square}
	For $f\in \mathcal{M}$, one has
	\begin{equation*}
	\frac{1}{|\mathcal{H}_{n}|}\sum_{D\in \mathcal{H}_{n}}\chi_{D}(f^{2})\; =\prod_{\substack{P\in\mathcal{P}\\P\mid f}}\left(1+\frac{1}{|P|} \right)^{-1} \; +\; O(|\mathcal{H}_n|^{-1}). 
	\end{equation*}
\end{lemma}
The following lemma strengthens \cite[Lemma~3.5]{BF2} by improving the $\varepsilon$-dependence to the quantity $|\ell_1|^{\varepsilon}$.
\begin{lemma}\label{character sum over nonsquare}
	Let $\ell=\ell_1 \ell_2^2\neq \square\in \mathbb{F}_q[t]$ with $\ell_1$ square-free. For any $\varepsilon>0$, we have
	\[
	\sum_{D\in \mathcal{H}_n}\chi_D(\ell)\ll  q^{(1/2+\varepsilon)n}\exp\left((d(\ell_1))^{1-\varepsilon}\right)\prod_{P\mid \ell_2}\left(1+|P|^{-1/2-\varepsilon}\right).
	\]	
	
\end{lemma}

\begin{proof}
	
	{\it {Bounding $\ln \mathcal{L}(u, \chi_D)$ using information about location of its zeros.}}
		Recall that for a given quadratic Dirichlet character $\chi$ with conductor $M \in \mathbb{F}_q[t]$, 
	\[
	\mathcal{L}(u, \chi) = \sum_{f \in \mathcal{M}} \chi(f) u^{\deg(f)}, \quad |u| < 1/q.
	\]
If $|u| < q^{-2}$, then $\ln \mathcal{L}(u, \chi) \ll 1$. 
For $\sigma_0 > \sigma > \tfrac{1}{2}$, consider the annulus 
$
q^{-2} < |u| \le q^{-\sigma}.
$
Observe that, since $\mathcal{L}(u, \chi)$ has no zeros in this region, $\ln \mathcal{L}(u, \chi)$ is analytic there. 
On the other hand, we know that
	\[
	\mathcal{L}(u, \chi) = \prod_{j=1}^{\deg(M)-1} \left(1 - u q^{1/2} e(-\theta_j)\right).
	\]
It follows that $
	|\mathcal{L}(u, \chi)| \le 2^{\deg(M)-1},
	$
	and consequently,
	\[
	|\ln \mathcal{L}(u, \chi)| \le 2 \deg(M).
	\]
	By the Borel--Carathéodory theorem, since $\ln \mathcal{L}(u, \chi)$ is analytic inside the closed disc $|u| \le q^{-\sigma_0}$, for all $|u| \leq q^{-\sigma}<q^{-\sigma_0}$, we get
	\[
	|\ln \mathcal{L}(u, \chi)| \le 
	\frac{2 q^{-\sigma}}{q^{-\sigma_0} - q^{-\sigma}} 
	\max_{|u| = q^{-\sigma_0}} \operatorname{Re}(\ln \mathcal{L}(u, \chi))
	+ \frac{q^{-\sigma_0} + q^{-\sigma}}{q^{-\sigma_0} - q^{-\sigma}}
	|\ln \mathcal{L}(0, \chi)|.
	\]
	This lead us to bound
	\begin{align}\label{bound for log}
	\ln \mathcal{L}(u, \chi) \ll \frac{\deg(M)}{q^{\sigma-\sigma_0} - 1}. 
	\end{align}
	{\it {Approximate $\ln \mathcal{L}(u, \chi)$ by a short Dirichlet polynomial.}}
	For $|u|\leq q^{-1-\varepsilon}$, one express 
	\[
	\ln \mathcal{L}(u, \chi) = \sum_{f \in \mathcal{M}} \frac{\Lambda(f) \chi(f)}{d(f)} u^{d(f)}.
	\]
	By Perron's formula \cite[page 79]{BF2}, for any $y \geq 2$,
	\[
	\sum_{\substack{f \in \mathcal{M}_{\leq y}}} \frac{\Lambda(f) \chi(f)}{d(f)} u^{d(f)} 
	= \frac{1}{2\pi i} \int_{|v| = R} \ln \mathcal{L}(uv, \chi) \frac{dv}{v^{y+1}(1 - v)},
	\]
where $R = q^{-\sigma_0}$ and $|u| = q^{-\sigma}$ for $\sigma > \sigma_0 \ge \tfrac{1}{2}$. 
Since $\mathcal{L}(u, \chi)$ has no zeros inside the disk $|v| \le q^{\sigma - \sigma_0}$, we may shift the contour of integration from $|v| = R$ to $|v| = R_1 = q^{\sigma - \sigma_0}$. 
In doing so, we cross a simple pole at $v = 1$, and by applying \eqref{bound for log}, we obtain
	\begin{align}\label{approx for log}
	\ln \mathcal{L}(u, \chi) &= 
	\sum_{\substack{f \in \mathcal{M}_{\leq y}}} \frac{\chi(f) \Lambda(f)}{d(f)} u^{d(f)}
	+ O \left( \int_{|v| = R_1} \ln \mathcal{L}(uv, \chi) \frac{dv}{v^{y+1}(1 - v)}\right) \nonumber\\
	&=\sum_{\substack{f \in \mathcal{M}_{\leq y}}} \frac{\chi(f) \Lambda(f)}{d(f)} u^{d(f)}+O \!\left( 
	\frac{\deg(M)}{q^{\sigma - \sigma_0}(q^{\sigma-\sigma_0}-1)}q^{-y(\sigma-\sigma_0)}
	\right).
	\end{align}
	{\it{Bounding character sum over non-square.}} Using Perron's formula again, we get
	\[
	\sum_{D\in \mathcal{H}_n}\chi_D(\ell)=\frac{1}{2\pi i}\int_{|u|=r}\mathcal{G}(u, \chi_\ell)\frac{du}{u^{n+1}},
	\]
	where we pick $r=q^{-1/2-\varepsilon}$ and 
	\[\mathcal{G}(u, \chi_\ell)=\sum_{\substack{D\in \mathcal{M}\\ D \text{ square-free}}} \chi_{\ell}(D)u^{d(D)}=\mathcal{L}(u, \chi_\ell)(1-qu^2)\prod_{P\mid \ell}\left(1-u^{2d(P)}\right)^{-1}.
	\]
	Given that $\ell=\ell_1 \ell_2^2$ with $\ell_1$ square-free, we express
	\[
	\mathcal{L}(u, \chi_{\ell})=\mathcal{L}(u, \chi_{\ell_1})\prod_{\substack{P\nmid \ell_1\\P\mid \ell_2}}\left(1-\chi_{\ell_1}(P)u^{d(P)}\right).
	\]
	Choosing $y=d(\ell_1)$, from \eqref{approx for log}, we find that for $|u|=q^{-1/2-\varepsilon}$,
	\[
	\mathcal{L}(u, \chi_{\ell_1})\ll \exp\left((d(\ell_1))^{1-\varepsilon}\right).
	\]
	From the above estimates, we conclude that
	\[
	\sum_{D\in \mathcal{H}_n}\chi_D(\ell)\ll_\varepsilon \int_{|u|=q^{-1/2-\varepsilon}}|\mathcal{G}(u, \chi_\ell)|\frac{|du|}{|u|^{n+1}}\ll q^{(1/2+\varepsilon)n} \exp\left((d(\ell_1))^{1-\varepsilon}\right)\prod_{P\mid \ell_2}\left(1+|P|^{-1/2-\varepsilon}\right),
	\]
	which finishes the proof.
\end{proof}

\subsection{Truncating the Euler product of $L(1, \chi_D)$}
We show that $L(1, \chi_D)$ can be well approximated by its short Euler product for almost all $D \in \mathcal{H}_n$. Let 
\begin{align}\label{short euler product}
L(1,\chi_D; M):= \prod_{d(P)\le M}\left(1-\frac{\chi_D(P)}{|P|}\right)^{-1}=\sum_{\substack{f\in \mathcal{M}\\ P\mid f \implies d(P)\leq M}}\frac{\chi_D(f)}{|f|}.
\end{align} 
	Let $M:=3\log n$. From \cite[Lemma 2.2]{Lumley}, for $D\in \mathcal{H}_n$, we have
	\begin{align}\label{asym for L(1)}
	L(1,\chi_{D})=L(1,\chi_D; M)\left(1+O\left(\frac{1}{n^{1/2}\log n}\right)\right).
	\end{align}
The following lemma provides a better approximation on the truncation length and the associated error chosen optimally for the proof of Theorem~\ref{main theorem}.
\begin{lemma}\label{full l to its truncation}
For sufficiently large $n$,
\[
L(1, \chi_D)=L(1, \chi_D; N)\left(1+O\left(\frac{1}{f(n)\log n}\right)\right)
\]
for all but at most $q^{\frac{7n}{10}}$ elements $D\in \mathcal{H}_n$, where $N=\log n+\log_2 n+3\log f(n)$, and $f(n) \to \infty$ arbitrarily slowly as $n \to \infty$, with $f(n) \ll \log n$.
	\end{lemma}
\begin{proof}
		By using \eqref{asym for L(1)}, we have
		\[
		L(1, \chi_D) = 
	L(1, \chi_D; N)
		\exp \!\bigg(
		\sum_{\substack{N < d(P) \le M}} 
		\bigg(\frac{\chi_D(P)}{|P|} 
		+ O\!\left( \frac{1}{|P|^2} \right)
		\bigg)\bigg)
		\bigg( 1 + O\!\bigg( \frac{1}{n^{\tfrac{1}{2}} \log n} \bigg) \bigg).
		\]
		Next, we identify a density zero subfamily of $\mathcal{H}_n$ for which
		\begin{align}\label{DP bound} 
		\Bigg|\sum_{N < d(P) \le M} \frac{\chi_D(P)}{|P|}\Bigg| \ge \frac{1}{h(n)}
		\end{align}
		holds, where $
		\log n \ll h(n) \to \infty$ as $n \to \infty.$
		It is worth noting that $h(n)$ must grow faster than $\log n$ chosen appropriately later in order for the saddle point analysis to yield an asymptotic formula exhibiting double exponential decay (see Section~\ref{spa}).

		In doing so, if \eqref{DP bound} holds for $q^{\theta n}$ many $D$’s for some $\theta < 1$, then we have
		\[
		L(1, \chi_D) = L(1, \chi_D; N) \left(1 + O\left(\frac{1}{h(n)} + \frac{1}{n^{1/2}\log n}\right)\right)
		\]
		for all but at most $q^{\theta n}$ many $D$’s.
		
		Dyadically partitioning the degrees of irreducible polynomials additively within the interval $[N, M]$, we apply Chebyshev’s inequality to obtain
		\begin{align}\label{chebyshev}
		\#\Bigg\{D \in \mathcal{H}_{n}: \bigg|\sum_{N_j < d(P) \le N_{j+1}} \frac{\chi_D(P)}{|P|}\bigg| 
		\ge \frac{1}{q^{j/4} h(n)}\Bigg\} 
		\le \sum_{D \in \mathcal{H}_{n}} 
		\bigg|\sum_{N_j < d(P) \le N_{j+1}} \frac{\chi_D(P)}{|P|}\bigg|^{2k_j}
		\left(q^{j/4} h(n)\right)^{2k_j},
		\end{align}
		where $k_j \ge 1$ will be chosen optimally later, and
		\[
		N_0:=N, \, \quad N_j = j + N, \qquad \text{ for }
		1\leq j \le \lfloor 3 \log n - N \rfloor.
		\]
	\noindent\textit{Large sieve type estimate.}  
	For given natural numbers $k, A$ and $B$, the discussion above forces us to estimate the following bound:
	\[
	S:= \sum_{D \in \mathcal{H}_{n}} 
	\bigg|\sum_{A < d(P) \le B} \frac{\chi_D(P)}{|P|}\bigg|^{2k} 
	\ll q^{\,n - kA} \left(\frac{k}{A}\right)^{k} 
	+ q^{(1/2 + \varepsilon)n} 
	\exp\!\left((2kB)^{1 - \varepsilon}\right) 
	\ln^{2k}\!\left(\frac{B}{A}\right).
	\]
	To prove this, we expand $2k$-th power to obtain the simplified expression  
		\[
		S = \sum_{P_1, \ldots, P_{2k}} \frac{1}{|P_1| \cdots |P_{2k}|} 
		\sum_{D \in \mathcal{H}_{n}} \chi_D(P_1 \cdots P_{2k}).
		\]
		Consider first the case when \( P_1 \cdots P_{2k} = \square \).  
		Using Lemma~\ref{character sum over square} and noting that, in the optimal situation, there are \( k \) distinct primes that can be paired, we deduce
		\[
		S \ll q^{n} \frac{(2k)!}{2^k k!}
		\bigg(\sum_{A < d(P) \le B} \frac{1}{|P|^2}\bigg)^{\!k}
		\ll q^{\,n - kA} \left(\frac{2k}{A}\right)^{\!k},
		\]
		where Stirling’s formula and \eqref{prime poly th} have been used.
		
		 Now consider the case \( P_1 \cdots P_{2k} \neq \square \).  
		Applying Lemma~\ref{character sum over nonsquare} together with the prime polynomial theorem (\eqref{prime poly th} repeatedly, we bound
		\begin{align*}
		S &\ll q^{(1/2 +\varepsilon)n} \exp\left((2kB)^{1-\varepsilon}\right) \bigg(\sum_{A < d(P) \le B} \frac{1}{|P|}\bigg(1+\frac{1}{|P|^{1/2}}\bigg)\bigg)^{2k}\\
		&\ll q^{(1/2+\varepsilon)n} \exp\left((2kB)^{1-\varepsilon}\right) \ln^{2k}\bigg(\frac{B}{A}\exp\bigg(O\bigg(\frac{1}{\ln A}\bigg)\bigg)\bigg),
		\end{align*}
	which establishes the desired bound for $S$. Substituting the contribution of the first term in the bound for $S$ into the moments on the right-hand side of the sum in~\eqref{chebyshev}, we have
	\begin{align*}
	 \eqref{chebyshev}&\ll q^{n-k_j N_j} \left(\frac{2k_j}{N_j}\right)^{k_j} \left(q^{j/4} h(n)\right)^{2k_j}\leq  q^{n-k_j N_j + \log \left(\frac{2k_j}{N_j}\right) + 2k_j\log \left(q^{j/4} h(n)\right)}\\
	&\le  q^{n+k_j(-\log f(n) - j/2)}\le q^{(1-\theta)n}
	\end{align*}
	 for some $\theta \in (0, 1)$, where we have chosen $h(n) = f(n) \log n$ and $
	 k_j = \frac{n\theta}{(j/4)+\log f(n)}$ with $f(n)\ll \log n$.
	 
By recalling the choice of $k_j$ and $h(n)$, the contribution from the second term of $S$ to \eqref{chebyshev} is estimated as 
	 \begin{align*}
	 \eqref{chebyshev}&\ll q^{(1/2+\varepsilon)n} \exp \left((2k_j N_j)^{1-\varepsilon}\right) \ln^{2k_j} \left(\frac{N_{j+1}}{N_j}\exp\left(O\left(\frac{1}{\ln N_j}\right)\right)\right) \left(q^{j/4} h(n)\right)^{2k_j}\\
	 &\ll q^{(1/2+\varepsilon)n}q^{(2k_j N_{j+1})^{1-\varepsilon}\log e-2k_j\log_2 n+\frac{jk_j}{2}+2k_j\log h(n)}\\
	 &\ll q^{(1/2+\varepsilon)n}q^{2k_j\left(j/4+\log f(n)\right)}\ll q^{(1/2+2\theta+\varepsilon)n},
	 \end{align*}
	 where we must take $\theta < \tfrac{1}{4}$ and note that 
	 \[ \frac{N_{j+1}}{N_j} e^{O\big(\frac{1}{\ln N_j}\big)}= \frac{j+1+N}{j+N} \left(1 + O\left(\frac{1}{\ln \log n}\right)\right)
	 \le  \left(1+\frac{1}{\log n}+O\left(\frac{1}{\ln \log n}\right)\right).
	 \]
	 Finally, choosing $\theta = \tfrac{1}{5}$ completes the proof of the lemma.
	\end{proof}

\section{Proof of Theorem \ref{main theorem}}
We begin by evaluating high moments of short Euler products, thereby extending the range of uniformity established in Lemma~3.5 of~\cite{Lumley}. This investigation is fundamentally linked to the expected value of the moments of random short Euler products (see equation~\eqref{realmoments}), which can be analyzed within the associated probabilistic framework (cf.~\cite[Section~4]{Lumley} and Section~\ref{mrsep}). The proof of the main theorem is then completed via a saddle point analysis, presented in Section~\ref{spa}. In what follows, we confine our attention to the computation of positive real moments, though the argument readily extends to complex moments of same strength.
\subsection{Moments for short Euler product}
For any real number $k \ge 1$, we have
\begin{align}\label{moment for SEP}
\sum_{D\in \mathcal{H}_n}L(1, \chi_D; y)^k=\sum_{\substack{f\in \mathcal{M}\\P\mid f \implies d(P)\leq y}}\frac{d_k(f)}{|f|}S(n; f),
\end{align}
where $d_k(f)$ is the generalized divisor function defined in \eqref{generalized_divisor}, and
\[
S(n; f) = \sum_{D \in \mathcal{H}_n} \chi_D(f).
\]
We decompose $f$ as $f = f_1 f_2^2 f_3^2$,
where $f_1$ and $f_2$ are square-free with $(f_1, f_2) = 1$, and $P\mid f_3 \implies P\mid f_1 f_2$.  
Observe that $
S(n; f_1 f_2^2 f_3^2) = S(n; f_1 f_2^2),
$ 
and hence the sum \eqref{moment for SEP} reduces to
\begin{align*}
&\sum_{f_1\in \mathcal{S}(y)}\mu^2(f_1)\sum_{\substack{f_1\in \mathcal{S}(y)\\ (f_1, f_2)=1}}\mu^2(f_2)S(n; f_1 f_2^2)\sum_{P\mid f_3 \implies P\mid f_1 f_2}\frac{d_k(f_1f_2^2 f_3^2)}{|f_1||f_2|^2|f_3|^2}\\
&=\sum_{f_1\in \mathcal{S}(y)}\mu^2(f_1)\sum_{\substack{f_1\in \mathcal{S}(y)\\ (f_1, f_2)=1}}\mu^2(f_2)S(n; f_1 f_2^2)\prod_{P\mid f_1}\mathcal{A}_{P}^-(k)\prod_{P\mid f_2}\left(\mathcal{A}_{P}^+(k)-1\right),
\end{align*}
where $\mathcal{S}(y)$ denotes the set of all monic polynomials whose prime divisors have degree $\leq y$, and
\[
\mathcal{A}_P^{\pm}(k)=\frac12 \left( \left(1-\frac{1}{|P|}\right)^{-k}\pm \left(1+\frac{1}{|P|}\right)^{-k}\right). 
\]
We now analyze the cases $\deg(f_1) = 0$ and $\deg(f_1) \ge 1$.  
If $\deg(f_1) = 0$, then $f_1 = 1$, since $f_1 \in \mathcal{M}$.  
Hence, by Lemma~\ref{character sum over square} and \eqref{expectation for sep}, the contribution to the above sum from the case $\deg(f_1) = 0$ is given by
\begin{align*}
	 &\sum_{f_2 \in \mathcal{S}(y)} \mu^2(f_2) S(h; f_2^2)
	\prod_{P \mid f_2} \big( \mathcal{A}_P^{+}(k) - 1 \big)\\
	&= |\mathcal{H}_n|\sum_{f_2 \in \mathcal{S}(y)} \mu^2(f_2)
	\prod_{P \mid f_2} \left( 1 + \frac{1}{|P|} \right)
	\prod_{P \mid f_2} \big( \mathcal{A}_P^{+}(k) - 1 \big) + O\!\bigg(
	\sum_{f_2 \in \mathcal{S}(y)}\mu^2(f_2)
	\prod_{P \mid f_2} |\mathcal{A}_P^{+}(k) - 1|
	\bigg)\\
	&= |\mathcal{H}_n|\, \mathbb{E}\!\left( L(1, \mathbb{X}; y)^k \right)
	+ O\!\bigg(
	\prod_{d(P) \le y} |\mathcal{A}_P^{+}(k) - 1|
	\bigg),
	\end{align*}
	where $\mathbb{E}\!\left( L(1, \mathbb{X}; y)^k\right)$ is given by \eqref{expectation for sep}. Next, we consider the case $\deg(f_1) \ge 1$.  
	Applying Lemma~\ref{character sum over nonsquare}, the corresponding contribution to the sum is bounded by
	\begin{align*}
	&\ll  q^{\left(\tfrac{1}{2}+\varepsilon+o(1)\right)n}
	\sum_{f_1 \in \mathcal{S}(y)} \sum_{\substack{f_2 \in \mathcal{S}(y)\\(f_1, f_2)=1}}
	\mu^2(f_1) \mu^2(f_2)
	\prod_{P \mid f_1}
	|\mathcal{A}_P^{-}(k)| \, |f_1|^{\frac{1}{n^\varepsilon}}
	\times \prod_{P \mid f_2} |\big( \mathcal{A}_P^{+}(k) - 1 \big)|
	\left( 1 + \frac{1}{|P|^{1/2 + \varepsilon}} \right)\\
	&\ll q^{(1/2+\varepsilon+o(1))n}
	\prod_{d(P) \le y}
	\left(1 + |\mathcal{A}_P^{-}(k)| |P|^{\frac{1}{n^\varepsilon}}
	+ |\mathcal{A}_P^{+}(k) - 1|
	\left(1+|P|^{-1/2-\varepsilon} \right)
	\right),
	\end{align*}
since $\exp\left((d(f_1))^{1-\varepsilon}\right)\leq q^{o(n)}|f_1|^{1/n^{\varepsilon}}$.	
We now divide the analysis into two parts according to whether
\[
d(P) \le \left\lfloor \frac{\log(4k + 4)}{\log q} \right\rfloor =: s_k
\quad \text{or} \quad
d(P) > \left\lfloor \frac{\log(4k + 4)}{\log q} \right\rfloor.
\]
In the latter case, we have
\[
|\mathcal{A}_P^{-}(k)| \le \frac{2k}{|P|}, 
\qquad 
|\mathcal{A}_P^{+}(k) - 1| \le \frac{2k^2}{|P|^2}.
\]
Hence, the contribution is bounded by
\begin{align*}
&\ll q^{(\frac{1}{2} + \varepsilon + o(1))n} 
\quad \times \prod_{d(P) \le \min\{y, s_k\}} 
\biggl(1 + |\mathcal{A}_P^{-}(k)|\, |P|^{1/n^\varepsilon} + 2\, |\mathcal{A}_P^{+}(k) - 1|\biggr) \\
&\quad \times \prod_{d(P) > \min\{y, s_k\}} 
\biggl(1 + 5\, |P|^{1/n^\varepsilon} \frac{k}{|P|}\biggr).
\end{align*}
	Note that for $k\geq 1$,
	\[
	|A_P^{+}(k)| \le \left(1-|P|^{-1}\right)^{-k},
	\]
	and from \cite[eq. $(3.5)$]{Lumley},
	\begin{align}\label{expectation for sep}
	\mathbb{E}\!\left( L(1, \mathbb{X}; y)^k \right)
	= \prod_{d(P) \le y}
	\left(
	\frac{1}{|P|+1}
	+ \frac{|P|}{2(|P|+1)}
	\left( \left(1 - \frac{1}{|P|}\right)^{-k}
	+ \left(1 + \frac{1}{|P|}\right)^{-k}
	\right)
	\right).
	\end{align}
	Consequently,
	 \[
	 \prod_{\min\{y, s_k\}}\left( \frac{|P|}{2(|P|+1)}\left(1-\frac{1}{|P|}\right)^{-k}\right)\leq \mathbb{E}\left(L(1, \mathbb{X}; y)^k\right).
	 \]
These observations allow us to reduce the previous bound to
	\begin{align*}
	\ll q^{(1/2+\varepsilon+o(1))n} \prod_{d(P) \leq \min\{y, s_k\}} |P|^{1/n^{\varepsilon}}\left(2 + |\mathcal{A}_P^+(k)|+ |\mathcal{A}_{P}^-(k)|\right) \prod_{d(P)>\min\{ y, s_k\}}\left(
	1 + 5|P|^{\frac{1}{n^\varepsilon}} \frac{k}{|P|}
	\right).
	\end{align*}
	If $y > s_k$, then this is bounded by
\begin{align*}
&\ll q^{(1/2+\varepsilon+o(1))n}\mathbb{E}\!\left( L(1, \mathbb{X}; y)^k \right)\prod_{d(P) \leq s_k}12 |P|^{1/n^{\varepsilon}}\prod_{d(P)> s_k}\left(
1 + 5|P|^{\frac{1}{n^\varepsilon}} \frac{k}{|P|}
\right)\\
&\ll  q^{(1/2+\varepsilon+o(1))n}\mathbb{E}\!\left( L(1, \mathbb{X}; y)^k \right)\exp\left(\frac{k\ln (12)}{\log(4k+4)}+O\left(\frac{k}{n^{\varepsilon}}\right)+5kq \ln \left(\frac{y}{\ln (4k+4)}\right)\right),
\end{align*} 
provided that $y\leq n^{\varepsilon}$. 
Choosing 
\[
k = \frac{n \log n}{q^4 B} \quad \text{with } B \ge 3q, \quad \text{and} \quad y = \log n + \log_2 n + B,
\] we obtain
	$$\ll q^{\left(\frac{1}{2}+\varepsilon+o(1)\right)n} \exp\left( \frac{(\ln 12) n}{q^4 B} + \frac{10 n}{q^3}\right) \mathbb{E}\!\left( L(1, \mathbb{X}; y)^k \right) 
	 \ll q^{\frac{9n}{10}}\mathbb{E}\!\left( L(1, \mathbb{X}; y)^k \right).
	 $$
	 Combining all estimates, we conclude that, letting $B\geq 3q$ and $y = \log n + \log_2 n + B$, uniformly for any $2 \le k \le \frac{n \log n}{q^4 B}$,
	 \begin{equation}\label{realmoments}
	 \frac{1}{|\mathcal{H}_n|} \sum_{D \in \mathcal{H}_n} L(1, \chi_D; y)^k
	 = \mathbb{E}\!\bigl( L(1, \mathbb{X}; y)^k \bigr) \Bigl( 1 + O\left(q^{-n/10}\right) \Bigr).
	 \end{equation}
	 In the next section, we provide an estimate for $\mathbb{E}\!\bigl( L(1, \mathbb{X}; y)^k \bigr)$.
	 \subsection{Moments of random short Euler product}\label{mrsep}
	For any $r, y \ge q$, define
	\[
	\mathcal{L}(r; y) := \log \mathbb{E}\!\bigl( L(1, \mathbb{X}; y)^r \bigr) = \sum_{d(P) \le y} \log E_P(r),
	\]
	where $E_P(r)$ is as in \cite[eqs.~(3.5) and (4.1)]{Lumley}.  
	
	Following the proof of \cite[Proposition~4.2]{Lumley}, one finds
	\begin{equation}\label{expectation at y}
	\mathcal{L}(r; y) = r \ln \log \min\{r, q^y\} + r \gamma + \frac{r G_1(tr/y)}{\log r} + O\!\left( \frac{r \log_2 r}{(\log r)^2} \right),
	\end{equation}
	where $G_1(t)$ is defined in \cite[eq.~(4.7)]{Lumley}, and $k \in \mathbb{Z}$ is the unique positive integer such that $q^k \le r < q^{k+1}$ with $t:= r/q^k$.  
	
	Moreover, the function $G_1$ satisfies
	\begin{equation}\label{tau and kappa relation at y}
	G_1(tr/y) = G_1(t) + O\!\bigl(r/q^y\bigr).
	\end{equation}
	 	\subsection{Saddle point analysis}\label{spa}
   For any $r \geq q$, define
	 	\begin{align*}
	 	\phi_n(\tau; y) &:=\frac{1}{|\mathcal{H}_n|}\sum_{\substack{D\in \mathcal{H}_n\\ L(1, \chi_D; y)\geq e^{\gamma}\tau}}1, \end{align*} 
	 	which gives
	 	\begin{align*} r\int_{0}^{\infty}u^{r-1}\phi_n(u; y)du
	 	&=\frac{e^{-\gamma r}}{|\mathcal{H}_n|}\sum_{D\in \mathcal{H}_n}L(1, \chi_D; y)^{r}.
	 	\end{align*}
	 	We apply \eqref{realmoments} for $q^{q}\leq f:=f(n)\leq \log n$ and $B=3\log f(n)$ to get 
	 	\begin{align*}
	 	&r \int_{0}^{\infty}u^{r-1}\phi_n(u; y)du=e^{-\gamma r} E\left(L(1, \mathbb{X}; y)^{r}\right)+O\left(e^{-\gamma r}q^{-n/10}  E\left(L(1, \mathbb{X}; y)^{r}\right)\right),
	 	\end{align*}
	 	uniformly for $r\leq \frac{n \log n}{3q^4 \log f(n)}$ and $y=\log n +\log_2 n+3\log f(n)$. Using \eqref{expectation at y} and \eqref{tau and kappa relation at y},
	 	we get
	 	\begin{align}\label{asymp for kappa2-short}
	 	\int_{0}^{\infty}u^{r-1}\phi_n(u; y)du 
	 	&=(\log r)^{r}\exp\left(\frac{r}{\log r}G_1(t)+O\left(\frac{r \log_2 r}{(\log r)^2}\right)+O\left(\frac{r^2}{q^y \log r}\right)\right).
	 	\end{align}
	 	Let  $r_2=r e^{\delta}$, where $\delta>0$ is sufficiently small to be determined later. We apply  \eqref{asymp for kappa2-short} with $r_2$ and rankin's trick to obtain  
	 	\begin{align*}
	 	&\int_{\tau+\delta}^{\infty}u^{r-1}\phi_n(u; y)du \leq (\tau +\delta)^{r-r_2}\int_{0}^{\infty}u^{r_2-1}\phi_n(u; y)du\\
	 	&=(\log{r} )^{r} \exp\left(r(1-e^{\delta})\log \left(1+\frac{\delta}{\tau}\right)+r e^{\delta}\log \left(1+\frac{\delta}{\log r}\right)\right) \\
	 	& \times \exp\left(\frac{r e^{\delta}}{\log r}G_1(t)+O\left(\frac{r\log_2 r}{(\log r)^2}\right)+O\left(\frac{r^2}{q^y\log r}\right)\right)  \exp \left(r (1-e^\delta) \left( \log{\tau} - \log_2{r} \right) \right), \end{align*}
	 	which holds uniformly for 
	 	\begin{align*} 
	 	r_2 \leq \frac{n \log n}{3q^4 \log f(n)} \iff r\leq  e^{-\delta}\frac{n \log n}{3q^4 \log f(n)}.
	 	\end{align*}
	 	Note that, from \cite[eq.~(4.11)]{Lumley}, 
	 	\[
	 	\log \kappa = \tau - G_2\!\bigl(q^{\{\log \kappa\}}\bigr) + O\!\left( \frac{\log \tau}{\tau} \right),
	 	\]
	 	since, for every $\tau$, there exists a unique $\kappa$ satisfying \cite[eq.~(4.2)]{Lumley}, where $G_2(t)$ is defined in \cite[eq.~(4.9)]{Lumley}.  
	 	Hence, for any such $\kappa$, we may choose $r = \kappa(\tau)$ via
	 	\begin{align}\label{g_2 function}
	 	\log \kappa = \tau - G_2\!\bigl(q^{\{\log \kappa\}}\bigr).
	 	\end{align}
	Choosing $\delta = \frac{c}{\sqrt{\log r}}$, from the above informations,
	 	\begin{align*}
	 	\int_{\tau+\delta}^{\infty}u^{r-1}\phi_n(u; y)du& \leq (\log r)^{r}\exp\left(\frac{r}{\log r}G_1(t)+O\left(\frac{r \log_2 r}{(\log r)^2}\right)+O\left(\frac{r^2}{q^y\log r}\right)\right) \exp \left( {-\frac{c^2r}{(\log r)^2}} \right)		\end{align*}
	 	for some  constant $c$. 
	 	Using \eqref{asymp for kappa2-short}, we have 
	 	\begin{align*}
	 	\int_{\tau+\delta}^{\infty}u^{r-1}\phi_n(u; y)du &\leq \left(\int_{0}^{\infty}u^{r-1}\phi_n(u; y)du\right) \exp \left( {-\frac{c^2 r}{(\log r)^2}} \right),
	 	\end{align*}
	 	and similarly, one can show that
	 	\begin{align*}
	 	\int_0^{\tau-\delta}u^{r-1}\phi_n(u; y)du\leq \left(\int_{0}^{\infty}u^{r-1}\phi_n(u; y)du\right) \exp \left( {-\frac{c^2r}{(\log r)^2}} \right).
	 	\end{align*}
	 	The exponential decay in the above inequalities implies that most of the mass of the integral is concentrated within the sector $[\tau - \delta,\, \tau + \delta]$ as $r$ increases.
	 	 Applying \eqref{asymp for kappa2-short}, 
	 	\begin{align*} 
	 	\int_{\tau-\delta}^{\tau+\delta}u^{r-1}\phi_n(u; y)du&= \int_0^{\infty} u^{r-1}\phi_n(u; y)du \left( 1 + O \left( \exp \left( {-\frac{c^2r}{(\log r)^2}} \right)\right) \right) \\
	 	&=(\log r)^{r}\exp\left(\frac{r}{\log r}G_1(t)+O\left(\frac{r \log_2 r}{(\log r)^2}\right)+O\left(\frac{r^2}{q^y\log r}\right)\right).
	 	\end{align*}
	 	Since $\phi_n(u; y)$ is non-increasing function, we obtain
	 	\[
	 	\tau^{r}\exp\left(O\left(\frac{\delta r}{\tau}\right)\right)\phi_n(\tau+\delta;y)\leq \int_{\tau-\delta}^{\tau+\delta}u^{r-1}\phi_n(u; y)du \leq \tau^{r}\exp\left(O\left(\frac{\delta r}{\tau}\right)\right)\phi_n(\tau-\delta; y).
	 	\]
	 Using \eqref{asymp for kappa2-short} and \eqref{g_2 function},  
	 	\begin{align*}
	 	\phi_n(\tau+\delta; y)\leq
	 	\exp\bigg(-C_0\left(q^{\{\log \kappa\}}\right)\frac{e^{\tau-C_1\left(q^{\{\log \kappa\}}\right)}}{\tau}\left(1+O\left(\delta\right)+O\left(e^{\tau}/q^y\right)\right)\bigg)
	 	\leq \phi_n(\tau-\delta; y),
	 	\end{align*}
	 	and therefore
	 	\begin{align}\label{distribution for phi-short}
	 	\phi_{n}(\tau; y)=\exp\bigg(-C_0\left(q^{\{\log \kappa\}}\right)\frac{e^{\tau-C_1\left(q^{\{\log \kappa\}}\right)}}{\tau}\left(1+O\left(\tau^{-\frac{1}{2}}\right)+O\left(e^{\tau}/q^y\right)\right)\bigg)
	 	\end{align}
	 	uniformly in the range
	 	\begin{align*}
	 	\tau&=\log r+C_0\left(q^{\{\log \kappa\}}\right)\leq \log \left( e^{-\delta}\frac{n \log n}{3 q^4 \log f(n)}
	 	\right) + C_0\left(q^{\{\log \kappa\}}\right)\leq \log n+\log_2 n-\theta(n),
	 	\end{align*}
	 where $\theta(n)=\log_2 f(n)-4-\log 3+o(1)\ll \log_3 n$. 
	 \begin{proof}[Proof of Theorem \ref{main theorem}]
	 We use Lemma \ref{full l to its truncation} to replace $\phi_n(\tau; y)$ by $\phi_n(\tau)$, with an appropriate error term. So,
	 	\begin{align*}
	 	\phi_{n}(\tau)=\frac{1}{|\mathcal{H}_n|}\sum_{\substack{D\in \mathcal{H}_n\\ L(1, \chi_D)\geq e^{\gamma}\tau}}1&=\frac{1}{|\mathcal{H}_n|}\sum_{\substack{D\in \mathcal{H}_n\\ L(1, \chi_D; y)\geq e^{\gamma}\tau \left(1+O\left(\frac{1}{f(n)\log n}\right)\right)}}1+O(q^{-\frac{3n}{10}})\\
	 	&= \phi_{n}\left(\tau\left(1+O\left(1/(f(n)\log n)\right)\right); y\right)+O(q^{-\frac{3n}{10}}),
	 	\end{align*}
	 	and from \eqref{distribution for phi-short}, we have
	 	\begin{align*}
	 	&\phi_{n}\left(\tau\left(1+O\left(1/(f(n)\log n)\right)\right); y\right)\\
	 	&=\exp\bigg(-C_0\left(q^{\{\log \kappa\}}\right)\frac{e^{\tau-C_1\left(q^{\{\log \kappa\}}\right)}}{\tau}\left(1+O\left(\tau^{-\frac{1}{2}}\right)+O\left((f(n))^{-1}\right)+O\left(e^{\tau}/q^y\right)\right)\bigg).
	 	\end{align*}
	 	Hence, we conclude that
	 	\begin{align*}
	 	\phi_{n}(\tau)&=\exp\bigg(-C_0\left(q^{\{\log \kappa\}}\right)\frac{e^{\tau-C_0\left(q^{\{\log \kappa\}}\right)}}{\tau}\left(1+O\left(\tau^{-\frac{1}{2}}\right)+O\left((f(n))^{-1}\right)\right)\bigg).
	 	\end{align*}
	This completes the proof of the theorem by taking $f(n)$ to be a slowly increasing function of $n$, so that $\theta(n)$ increases accordingly.
	\end{proof}	
\section{Proof of Theorem \ref{theorem on resonance method}}
We employ the long resonator method applied by Aistleitner \emph{et al.}~\cite{AMM, AMMP} and Darbar--Maiti~\cite{DM}, based on the original idea of Bondarenko--Seip~\cite{BS, BS3}.
\subsection{Resonator and related estimates}
 Let us consider $N:=\log n+ \log_2 n+\log c$, where the constant $c$ will be optimized later.
We define the resonator as
\[R_{D}=\prod_{d(P)< N} \left(1-\left(1-\frac{|P|}{q^N}\right)\chi_{D}(P)\right)^{-1}=\sum_{f\in \mathcal{M}} r_f \chi_{D}(f),
\] 
where $r_{{P}} =(1-|P|/q^N)$ for $d(P)< N$, $r_P=0$ for $d(P)\geq N$. The coefficients $r_f$ are extended completely multiplicatively.
We start with bounding
\begin{align*}|R_{D}|=\prod_{d(P)< N} \big|\left(1-r_P \chi_{D}(P)\right)^{-1}\big| \le \prod_{d(P)< N} \left( 1-\left(1- \frac{|P|}{q^N}\right) \right)^{-1}\le \prod_{d(P)\leq N}\frac{q^N}{|P|}.
\end{align*} 
Applying the prime polynomial theorem and Abel’s summation formula, 
\begin{align}\label{RD}
&\log |R_{D}|\le N\sum_{d(P)\leq N} 1 -\sum_{d(P)\leq N} d(P)
= N\Pi_{q}(N)-\sum_{n\leq N}n\pi_{q}(n)\nonumber\\
&=\int_{1}^{N}\Pi_q(u)du=\frac{c \zeta_{\mathbb{A}}(2)}{\ln q}\left(1+O\bigg(\frac{\log_2 n}{\log n}\bigg)\right)n,
\end{align}
since $\Pi_{q}(N)=\zeta_{\mathbb{A}}(2)\frac{q^N}{N}\left(1+O\left(\frac{\ln N}{N}\right)\right)$ and the choice of $N$.

We estimate the mean square of the resonator using Lemmas~\ref{character sum over square} and~\ref{character sum over nonsquare} to obtain
\begin{align*}
&S_2:=\sum_{\substack{D\in \mathcal{H}_n}}R_{D}^{2}=\sum_{f, g\in \mathcal{M}}r_f r_g \sum_{\substack{D\in \mathcal{H}_n}}\chi_{D}(fg)\\
&=|\mathcal{H}_n|\sum_{\substack{f, g \in \mathcal{M}\\ fg=\square }}r_f r_g h(fg)  
 +O\Bigg(q^{(\frac 1 2 + \ve)n} \sum_{\substack{f, g\in \mathcal{M}\\fg=\ell_1 \ell_2^2}}r_f r_g \exp\left((d(\ell_1))^{1-\varepsilon}\right)\prod_{P\mid \ell_2}(1+|P|^{-1/2-\varepsilon})\Bigg),	
\end{align*}
where $h(f)=\prod_{P|f} \left(\frac{|P|}{|P|+1}\right)$.

 Since $\ell_1$ is a square-free part of $fg$ and $\ell_1\mid fg$, it implies from \eqref{prime poly th} that $d(\ell_1)\le \sum_{d(P)\le N} d(P) = \zeta_{\mathbb{A}}(2)q^N(1+o(1))$. Therefore, with the choice of $N$, 
\[
\exp\left(d(\ell_1)^{1-\varepsilon}\right)=\exp\left(\left(\zeta_{\mathbb{A}}(2)q^N\right)^{1-\ve}\right)\ll q^{\ve n}.
\] 
Using these estimates together with~\eqref{RD}, we obtain that the error term is of size $$O\left( q^{\left(\frac12+\frac{2c\zeta_{\mathbb{A}}(2)}{\ln q}+2\varepsilon\right)n}\right).$$
Hence, by the multiplicativity of $r_f$ and noting that $h(p^k)=h(k)$, we conclude 
\begin{align*}
S_2
=|\mathcal{H}_n|\sum_{\substack{k\in \mathcal{M}} } r^2_k\, d(k^2) h(k)+O\left( (q^{\left(\frac12+\frac{2c\zeta_{\mathbb{A}}(2)}{\ln q}+2\varepsilon\right)n}\right).
\end{align*}
\subsection{Interaction between the $L$-function and the resonator over the family}
Using Lemmas \ref{character sum over square} and \ref{character sum over nonsquare}, we compute
\begin{align*}
S_1&=\sum_{D\in \mathcal{H}_n}L(1,\chi_D; M) R_{D}^{2}=\sum_{f, g\in \mathcal{M}}r_f r_g \sum_{e\in \mathcal{M}}\frac{a_e}{|e|} \sum_{\substack{D\in \mathcal{H}_n}}\chi_{D}(efg)\\
&=\sum_{e\in \mathcal{M}}\frac{a_e}{|e|}\sum_{\substack{fg\in \mathcal{M}\\efg=\square}}r_f r_g \Bigg(|\mathcal{H}_n|\prod_{\substack{P\mid efg}} \left(\frac{|P|}{|P|+1}\right) + O(1)\Bigg)\\
&+O\Bigg(q^{\left(\frac12 + \ve\right) n}\sum_{e\in \mathcal{M}}\frac{a_e}{|e|} \sum_{\substack{f, g\in \mathcal{M}\\efg=s_1 s_2^2}}r_f r_g \alpha(s_1)\beta(s_2)\Bigg),
\end{align*}
where $\alpha(s_1):=\exp\left((d(s_1))^{1-\varepsilon}\right)$ and $\beta(s_2):=\prod_{P\mid s_2}\left(1+|P|^{-1/2-\varepsilon}\right)$. 

The first error term of the above expression is bounded above by
$$\ll  \left(\sum_{e\in \mathcal{M}}\frac{a_e }{|e|}\right) \left(\sum_{\substack{f\in \mathcal{M}}}r_f\right)^2\ll q^{\frac{2 c \zeta_{\mathbb{A}}(2)}{\ln q}\left(1+O\left(\log_2 n/\log n\right)\right)n}\log n,$$
since for $M= 3\log n$,
$$\displaystyle\sum_{e\in \mathcal{M}}\frac{a_e}{|e|}\ll \prod_{d(P)\le M}\left(1+\frac{1}{|P|}\right)\ll M\ll \log n,$$ 
and following \eqref{RD} and the choice of $N$,  
\begin{align}\label{rmgm}
\displaystyle\sum_{\substack{f\in \mathcal{M}}} r_f & =\prod_{d(P)< N}\left(1-r_P\right)^{-1}=q^{\frac{c\zeta_{\mathbb{A}}(2)}{\ln q}\left(1+O\left(\log_2 n/\log n\right)\right)n}.
\end{align} 
 The second error term is bounded above by
\begin{align}\label{end et}
\ll q^{\left(\frac12 + \ve\right) n}\Bigg(\sum_{\substack{e\in \mathcal{M}\\e=e_1 e_2^2}}\frac{a_e \alpha(e_1)\beta(e_2)}{|e|}\Bigg) \Bigg(\sum_{\substack{f\in \mathcal{M}\\f=f_1f_2^2}}r_f \alpha(f_1)\beta(f_2)\Bigg)^2.
\end{align}
Using \eqref{prime poly th}, the choice of $M$, and $\alpha(e_1)\ll q^{\varepsilon n/4}$, we compute that 
\begin{align*}
&\sum_{e\in \mathcal{M}}\frac{a_e \alpha(e_1)\beta(e_2)}{|e|}\le \sum_{d(e)\le \log n}\frac{a_e \alpha(e_1)}{|e|} + \sum_{d(e)> \log n}\frac{a_e \alpha(e_1)\beta(e_2)}{|e|}\\
&	\le q^{\varepsilon n/4}\sum_{e\in \mathcal{M}}\frac{a_e \beta(e_2)}{|e|} + \sum_{d(e)> \log n}\frac{a_e \beta(e_2) |e_1|^{\frac{1}{n^{\ve}}}}{|e|}\ll q^{\frac{\ve n}{3}}.
\end{align*}
From \eqref{rmgm} and above estimation together with the bound $\alpha(f_1)\ll q^{\ve n/4}$, we get
\[
\eqref{end et}\ll  q^{\left(\frac12+\frac{2c\zeta_{\mathbb{A}}(2)}{\ln q}+\frac{11\varepsilon}{6}\right)n}.
\]
Hence, we obtain
\begin{align*}
S_1=|\mathcal{H}_n| \sum_{e\in \mathcal{M}}\frac{a_e}{|e|}\sum_{\substack{e, f, g\in \mathcal{M}\\ efg=\square }}r_f r_g h(efg)  + O\left(q^{\left(\frac12+\frac{2c\zeta_{\mathbb{A}}(2)}{\ln q}+\frac{11\varepsilon}{6}\right)n}\right).
\end{align*}
\subsubsection{Shortening the length of Euler product using positivity}
Consider $L(1,\chi_D; N):= \prod_{d(P)\le N}$$\left(1-\frac{\chi_D(P)}{|P|}\right)^{-1} = \sum_{f\in \mathcal{M}}\frac{b_f \chi_{D}(f)}{|f|}$, where $b_f =1$ if all prime divisors of $k$ have degree $\le N$, and $0$ otherwise. Here, we employ the positivity of resonator coefficients, namely $r_f\geq 0$, alongside ensuring that $a_f\geq b_f\geq 0$, enabling us to disregard irreducible polynomials that have degree larger than $N$. Consequently, we arrive at
\begin{align*}
S_1\ge |\mathcal{H}_n| \sum_{e\in \mathcal{M}}\frac{b_e}{|e|}\sum_{\substack{e, f, g\in \mathcal{M}\\ efg=\square }}r_f r_g h(efg)  + O\left(q^{\left(\frac12+\frac{2c\zeta_{\mathbb{A}}(2)}{\ln q}+\frac{11\varepsilon}{6}\right)n}\right).
\end{align*}
\begin{proof}[Proof of Theorem \ref{theorem on resonance method}]
	We now aim to obtain a lower bound for the ratio $S_1/S_2$ based on the preceding estimates of $S_1$ and $S_2$, optimized with respect to the parameter $c$.  
	The optimization procedure is carried out in the Appendix.
	
	From the computation presented therein, together with~\eqref{ratio} and the choice 
	$
	N = \log n + \log_2 n + \log c,
	$
	we deduce that
\begin{align}\label{ratio final}
\frac{S_1}{S_2}\geq e^{\gamma}\left(\log n+\log_2 n+ \frac12-c_3 \frac{q}{q-1}+\log c+O\left(\frac{\log_2 n}{\log n}\right)\right),
\end{align}
where $c_3=\frac{\pi}{4}-\frac{\ln 2}{2}$ and $\zeta_{\mathbb{A}}(2)=\frac{q}{q-1}$. 
Also, from \eqref{optimzed constant}, we can choose $c$ in such a way that 
\begin{align}\label{constant c}
\log c=\log\left(\frac{\ln q}{2(3\ln 2-\pi/2)\zeta_{\mathbb{A}}(2)}\right)-\beta+\frac{1}{\sqrt{\log n}}.
\end{align}

The additional factor $1/\sqrt{\log n}$ in the choice of $c$ is essential for proving Theorem \ref{theorem on resonance method}.
Hence, using \eqref{asym for L(1)}, we can that there exists a $D\in \mathcal{H}_n$ such that 
\[
L(1,\chi_D)\ge e^{\gamma}\left(\log n +\log_2 n+C_2(q)-\beta+o(1)\right).
\]
We use \eqref{asym for L(1)}, \eqref{ratio final}, \eqref{constant c} to obtain 
\begin{align*}
&\bigg(\tau_{\beta, n}+\frac{1}{2\sqrt{\log n}}\bigg)\sum_{D\in \mathcal{H}_n}R_D^2\leq \sum_{D\in \mathcal{H}_n}L(1, \chi_D)R_D^2\\
&= \sum_{\substack{D\in \mathcal{H}_n\\ L(1, \chi_D)< \tau_{\beta, n}}}L(1, \chi_D)R_D^2+\sum_{\substack{D\in \mathcal{H}_n\\ L(1, \chi_D)\geq  \tau_{\beta, n}}}L(1, \chi_D)R_D^2< \tau_{\beta, n} \sum_{D\in \mathcal{H}_n}R_D^2+ \sum_{\substack{D\in \mathcal{H}_n\\ L(1, \chi_D)\geq  \tau_{\beta, n}}}L(1, \chi_D)R_D^2.
\end{align*}
From the above computation of $S_2$ together with \eqref{scx} gives us
\[
\sum_{\substack{D\in \mathcal{H}_n\\ L(1, \chi_D)\geq \tau_{\beta, n}}}L(1, \chi_D)R_D^2> \frac{e^{\gamma}}{2 \sqrt{\log n}}\sum_{D\in \mathcal{H}_n}R_D^2> \frac{e^{\gamma}}{2}\frac{q^{(1+c')n\left(1+O(\log_2 n/\log n)\right)}}{\sqrt{\log n}}.
\]
On the other hand, using \cite[Proposition 1.4]{Lumley} that
\[
\max_{D\in \mathcal{H}_n}L(1, \chi_D)\leq 2e^{\gamma}\log n+O_q(1),
\]
we have from \eqref{RD},
\begin{align*}
\sum_{\substack{D\in \mathcal{H}_n\\ L(1, \chi_D)\geq  \tau_{\beta, n}}}L(1, \chi_D)R_D^2&\leq |\mathcal{H}_n| \left(\max_{D\in \mathcal{H}_n} R_D^2 \cdot \max_{D\in \mathcal{H}_n}L(1, \chi_D)\right)\phi_n(\beta)\\
&\leq \left(2e^{\gamma}\log n\right) q^{\left(1+\frac{2c\zeta_{\mathbb{A}}(2)}{\ln q}\right)\left(1+O\left(\frac{\log_2 n}{\log n}\right)\right)n} \phi_n(\beta).
\end{align*}
Comparing the above lower and upper bound, and using $c$ as expressed in \eqref{constant c}, we finally conclude that
\[
\phi_n(\beta)> \frac{1}{4(\log n)^{3/2}}q^{-\frac{c(3\ln 2-\pi/2)\zeta_{\mathbb{A}}(2)}{\ln q}\left(1+O\left(\log_2 n/\log n\right)\right)}\geq e^{-\frac{q^{-\beta}\left(1+O\left(1/\sqrt{\log n}\right)\right)\ln q}{2}}.
\]
\end{proof}

\section*{Acknowledgment} 
The author thanks Allysa Lumley and Igor Shparlinski for reading the first draft of this article, and Winston Heap for helpful discussions.  
The author also acknowledges the financial support of the Australian Research Council through Grant~DP230100534, and the hospitality of the Max Planck Institute for Mathematics, Bonn.


\section{Appendix}
\subsection{Estimating the ratio \texorpdfstring{$S_1/S_2$}{S1/S2} and optimizing the constant \texorpdfstring{$c$}{c}}
Our goal in this section is to derive an explicit expression for the sums involved in order to determine the optimal value of the constant~$c$.
\begin{align}\label{eq M}
\mathcal{S}_{c, N}:=\sum_{e\in \mathcal{M}}\frac{b_e}{|e|}\sum_{\substack{e, f, g\in \mathcal{M}\\ efg=\square }}r_f r_g h(efg) \quad \text{ and } \quad \mathcal{R}_{c, N}:=\sum_{\substack{k\in \mathcal{M}} } r^2_k\, d(k^2) h(k).
\end{align}
\subsubsection*{Decomposition of polynomials} We decompose $e=e_1 e_2^2 e_3^2$ where $e_1$ and $e_2$  are square-free, with $(e_1, e_2) = 1$, and $P|e_3 \implies P|e_1 e_2$.  Similarly, we decompose $f$ and $g$ as $f=f_1f_2^2f_3^2$ and $g=g_1g_2^2g_3^2$, where $f_i$ and $g_i$ $(i=1, 2)$ are  square-free, $(f_1, f_2)=(g_1, g_2)=1$, and for any prime polynomial $P$: $P\mid f_3\implies P\mid f_1 f_2$, and $P\mid g_3\implies P\mid g_1 g_2$. 

So $efg=\square \implies e_1 f_1 g_1 =\square$, where $e_1 , f_1,$ and $ g_1$ are square-free polynomials.

Without loss of generality, let $V=(e_1 , g_1)$. Write $e_1 = V e_1^{\prime}$ and $g_1 = V g_1^{\prime}$ with $(e_1^{\prime}, g_1^{\prime})=1$. Therefore $e_1 f_1 g_1=\square \implies f_ 1=  e_1^{\prime} g_1^{\prime}$. 
\subsubsection*{Euler product representation of $\mathcal{S}_{c, N}$} From the above decomposition, $\mathcal{S}_{c, N}$ becomes
\begin{align*}
\mathcal{S}_{c, N}&=\sum_{\substack{V \in \mathcal{M}\\ (V\,  e_1^{\prime}  g_1^{\prime} e_2 g_2)=1}} \frac{ \mu^2(V) b_V r_V }{|V|} \sum_{\substack{ e_1^{\prime}\in \mathcal{M}\\ ( e_1^{\prime},\,   g_1^{\prime} e_2 f_2)=1}} \frac{  \mu^2( e_1^{\prime}) b_{e_1^{\prime}} r_{e_1^{\prime}}}{|e_1^{\prime}|}  \sum_{\substack{ g_1^{\prime}\in \mathcal{M}\\ (g_1^{\prime},\, g_2 f_2)=1}} \mu^2(g_1^{\prime}) r_{ g_1^{\prime}}^2  \\
&\times \sum_{e_2 \in \mathcal{M}}\frac{\mu^2(e_2) b_{e_2 }}{|e_2|^2}   \sum_{f_2 \in \mathcal{M}}\mu^2(f_2) r_{f_2}^2 
\sum_{g_2 \in \mathcal{M}}\mu^2(g_2) r_{g_2}^2 h(V e_1^{\prime} g_1^{\prime} e_2 f_2 g_2)\\ 
&\times \sum_{\substack{e_3\in \mathcal{M} \\P|e_3  \implies P|V e_1^{\prime}  e_2 }} \frac{b_{e_3}  }{|e_3|^2 } \sum_{\substack{f_3\in \mathcal{M}\\ P\mid f_3 \implies P| e_1^{\prime} g_1^{\prime} f_2 }} r_{f_3}^2  \sum_{\substack{g_3\in \mathcal{M} \\P|g_3 \implies P|V g_1^{\prime} g_2 }} r_{g_3}^2.
\end{align*}
Define, $B_P =\left(1-\frac{1}{|P|^2}\right)^{-1}$ and $R_P =\left(1-r_{P}^2\right)^{-1}$. We express the following sums into the Euler product form
\begin{align*}
\sum_{\substack{e_3\in \mathcal{M}\\ P|e_3  \implies P|V e_1^{\prime} e_2 }} \frac{b_{e_3} }{|e_3|^2}= \prod_{\substack{P\mid V   e_1^{\prime} e_2  }} B_P,    \sum_{\substack{f_3\in \mathcal{M}\\ P|f_3 \implies P| e_1^{\prime} g_1^{\prime} f_2 }} r_{f_3}^2 =\prod_{\substack{P\mid  e_1^{\prime} g_1^{\prime} f_2}} R_P,    \sum_{\substack{g_3\in \mathcal{M} \\P|g_3 \implies P| V g_1^{\prime} g_2 }} r_{g_3}^2 =\prod_{\substack{P\mid  V g_1^{\prime} g_2 }} R_P.
\end{align*}
Therefore, 
\begin{align*}
\mathcal{S}_{c, N}&=\sum_{\substack{V \in \mathcal{M}\\ (V,\,  e_1^{\prime}  g_1^{\prime} e_2 g_2)=1}} \frac{ \mu^2(V)b_V r_V }{|V|} \sum_{\substack{ e_1^{\prime}\in \mathcal{M}\\ ( e_1^{\prime},\,   g_1^{\prime} e_2 f_2)=1}} \frac{\mu^2(e_1^{\prime}) b_{e_1^{\prime}} r_{ e_1^{\prime}}}{ |e_1^{\prime}|}  \sum_{\substack{ g_1^{\prime}\in \mathcal{M}\\ ( g_1^{\prime},\,    g_2 f_2)=1}} \mu^2( g_1^{\prime}) r_{ g_1^{\prime}}^2  \sum_{e_2\in \mathcal{M} }\frac{\mu^2(e_2) b_{e_2 }}{|e_2|^2} \\  
&\times \sum_{f_2 \in \mathcal{M}}\mu^2(f_2) r_{f_2}^2 
\sum_{g_2 \in \mathcal{M}}\mu^2(g_2) r_{g_2}^2 h(V e_1^{\prime} g_1^{\prime} e_2 f_2 g_2) \prod_{\substack{P\mid V   e_1^{\prime}  e_2  }} B_P \prod_{\substack{P\mid  e_1^{\prime} g_1^{\prime} f_2  }} R_P \prod_{\substack{P\mid  V g_1^{\prime} g_2 }} R_P .
\end{align*}
By multiplicativity this transform into the Euler product of the form 
\begin{align*}
\mathcal{S}_{c, N} &=\prod_{d(P)< N} \left( 1+ \frac{2r_P}{|P|} B_P R_P  h(P) + \frac{2r_P^3}{|P|} B_P R_P^2 h(P) + r_P^2 R_P^2 h(P) + \frac{r_P^2}{|P|^2} B_P R_P^2 h(P)    \right.\\
&\quad \quad\quad \quad  \left.  +\frac{B_P h(P)}{|P|^2} + \frac{2r_P^2}{|P|^2} B_P R_P h(P) + \frac{r_P^4}{|P|^2}B_P R_P^2 h(P) + 2r_P^2 R_P h(P) +r_P^4 R_P^2 h(P)   \right) \\
&=\prod_{d(P)< N}  B_P R_P^2 h(P) \left( 1+r_P^2+ \frac{2 r_P}{|P|} +\frac{1}{|P|} \left(1-r_P^2\right)^2 \left(1-\frac{1}{|P|^2}\right)\right)\\
&=\mathcal{E}(N) \mathcal{E}_1(N)\prod_{d(P)< N}\left(1-\frac{1}{|P|} \right)^{-1} h(P) ,
\end{align*}
where $\mathcal{E}(N):=\prod_{d(P)< N}\left(1-r_{P}^2 \right)^{-2} \left(1+r_{P}^2 \right)$, and
\begin{align*}
 \mathcal{E}_1(N):=\prod_{d(P)< N}\left(1- \frac{(1-r_P)^2}{(|P|+1)(1+r_P^2)}+\left(1-\frac{1}{|P|}\right)\frac{(1-r_P^2)^2}{|P|(1+r_P^2)} \right).
\end{align*}
\subsubsection*{Euler product representation of $\mathcal{R}_{c, N}$}
On the other hand, since $d(P^k)=k+1$ for all $k\geq 1$, we can express
\begin{align*}
\mathcal{R}_{c, N}&=\prod_{d(P)<N}\left(1+3r_P^2 h(P)+5r_P^4 h(P)+\cdots\right)=\prod_{d(P)<N}\left(1-r_P^2 \right)^{-2} \left(1+ \frac{|P|-2}{|P|+1} r_P^2 + \frac{r_P^4}{|P|+1}   \right)\\
&=\mathcal{E}(N)\mathcal{E}_2(N),
\end{align*}
where
$
\mathcal{E}_2(N):=\prod_{d(P)<N}\left(1-\frac{3r_P^2}{(|P|+1)\left(1+r_P^2\right)}+\frac{r_P^4}{(|P|+1)\left(1+r_P^2\right)}\right).
$

From the estimation of $\mathcal{S}_{c, N}$ and $\mathcal{R}_{c, N}$, we finally obtain 
\begin{align}\label{ratio}
\frac{S_1}{S_2}\ge \frac{|\mathcal{H}_n|   \mathcal{E}(N) \mathcal{E}_1(N) \prod_{d(P)< N}  (1-|P|^{-1}) + O\left(q^{\left(\frac12+\frac{2c\zeta_{\mathbb{A}}(2)}{\ln q}+\frac{11\varepsilon}{6}\right)n}\right)}{|\mathcal{H}_n| \mathcal{E}(N) \mathcal{E}_2(N)\prod_{d(P)<N}h(P)^{-1}  + O\left( (q^{\left(\frac12+\frac{2c\zeta_{\mathbb{A}}(2)}{\ln q}+2\varepsilon\right)n}\right)}.
\end{align}
 Noting $N=\log n+\log_2 n +\log c$, we first optimize $c$ from the ratio to the right-hand side of \eqref{ratio} by comparing the error terms. To do this, we first calculate the size of $\mathcal{E}(N)=\prod_{d(P)<N}\left(1-r_P^2 \right)^{-2} \left(1+r_P^2\right)$. Applying Abel's summation formula,
\begin{align*}
&\ln \mathcal{E}(N)=-2\sum_{n\leq N}\Pi_{q}(n)\ln\left(\tfrac{2q^n}{q^N}-\tfrac{q^{2n}}{q^{2N}}\right)
+\sum_{n\leq N}\Pi_{q}(n)\ln\left(2-\tfrac{2q^n}{q^N}+\tfrac{q^{2n}}{q^{2N}}\right)\\
&= \frac{\zeta_{\mathbb{A}}(2)}{\ln q}\left(2(2-\ln 4)+(\ln 2+\pi/2-2)\right)\frac{q^N}{N}\left(1+O\left(\frac{\log N}{N}\right)\right),
\end{align*}
where the first and second integrals are referenced from \cite[page $841$]{AMMP} and \cite[page $20$]{DM} respectively by using the prime polynomial theorem and change of variable $q^t=u$. So, from the choice of $N$, we have
\begin{align*}
\mathcal{E}(N)=q^{\widetilde{c}\left(1+O\left(\frac{\log_2 n}{\log n}\right)\right)n},
\end{align*}
where $\widetilde{c}=\left(c(2+\pi/2-3\ln 2)\zeta_{\mathbb{A}}(2)\right)/\ln q$.
Noticing that $\mathcal{E}_2(N)\asymp 1/n$, from \eqref{ratio}, we compare the first term of the denominator to the both error terms to restrict ourself
\begin{align}\label{optimzed constant}
\frac12+\frac{2c\zeta_{\mathbb{A}}(2)}{\ln q}< 1+\frac{c\left(2+\pi/2-3\ln 2\right)\zeta_{\mathbb{A}}(2)}{\ln q}\implies c<\frac{\ln q}{2(3\ln 2-\pi/2)\zeta_{\mathbb{A}}(2)}.
\end{align}
Also, from the above computations, we obtain
\begin{align}\label{scx}
\mathcal{S}_{c, N}=q^{n\widetilde{c}\left(1+O(\log_2 n/\log n)\right)}.
\end{align}
To extract the main term from the ratio on the right-hand side of \eqref{ratio}, it remains to estimate $\mathcal{E}_1(N)$, $\prod_{d(P)\le N} \left(1-|P|^{-1}\right)$, and $\mathcal{E}_2(N)\prod_{d(P)\le N} h(P)^{-1} $.  
Using the classical estimate using the Euler--Maclaurin summation formula,
\[
\sum_{n\le N} \frac{1}{n} = \ln N + \gamma + \frac{1}{2N} + O\Bigl(\frac{1}{N^2}\Bigr),
\]
one readily obtains
\[
\prod_{d(P)\le N} \left(1-\frac{1}{|P|}\right)^{-1} = e^{\gamma} N \left(1 + \frac{1}{2N} + O\Bigl(\frac{1}{N^2}\Bigr)\right).
\]
Similarly, following the approach used above for evaluating Euler products, we find
\[
\mathcal{E}_1(N) = q^{\left( \left(c_2 - c_3 + O\bigl(\frac{\log N}{N}\bigr)\right) \frac{\zeta_{\mathbb{A}}(2)}{N \ln q} \right)},
\]
where the constants $c_2$ and $c_3$ are as in \cite[page 21]{DM}, and
\[
\mathcal{E}_2(N)\prod_{d(P)\le N} h(P)^{-1} = \prod_{d(P)\le N} \left( 1 + \frac{(1-r_P^2)^2}{|P|(1+r_P^2)} \right) 
= q^{\left( \left(c_2 + O\bigl(\frac{\log N}{N}\bigr)\right) \frac{\zeta_{\mathbb{A}}(2)}{N \ln q} \right)}.
\]


\begin{thebibliography}{111}

\bibitem{AT}	Altug, S. A; Tsimerman, J. {\emph  Metaplectic Ramanujan conjecture over function fields with applications to quadratic forms}, {\it Int. Math. Res. Not. IMRN}, 2014, no. 13, 3465--3558.
	


\bibitem{AMM}
Aistleitner, C.; Mahatab, K.; Munsch, M. {\emph Extreme values of the Riemann zeta function on the $1$-line}, {\it Int. Math. Res. Not.}, no. {\bf 22} (2019), 6924--6932.

\bibitem{AMMP}
Aistleitner, C.; Mahatab, K.; Munsch, M.; Peyrot, A. {\emph On large values of $L(\sigma, \chi)$}, {\it Q. J. Math.}, {\bf 70} (2019), 831--848.

\bibitem{AK} 
Andrade, J. C.; Keating, J. P. {\emph The mean value of $L(\frac{1}{2},\chi)$ in the hyperelliptic ensemble}, {\it J. Number Theory}, {\bf 132} (2012), 2793--2816.

\bibitem{Artin} Artin, E. {\emph Quadratische Korper in Geibiet der H\"{o}heren Kongruzzen I and II}, {\it Math. Z.}, {\bf 19} (1924),
153--296.


	
\bibitem{BS} Bondarenko, A.; Seip, K. {\emph Large greatest common divisor sums and extreme values of the Riemann zeta function}, {\it Duke Math. J.}, {\bf 166} (2017), 685--701.
	
\bibitem{BS3}
Bondarenko, A.; Seip, K. {\emph Note on the resonance method for the Riemann zeta function. 50 years with Hardy spaces}, {\it Oper. Theory Adv. Appl.}, {\bf 261} (2018), 121--139.
	
	
\bibitem{BF} Bui, H; Florea, A. {\emph Zeros of quadratic Dirichlet L-functions in the hyperelliptic ensemble}, {\it Transactions of the AMS}, 2018, no. 11, 8013--8045.

\bibitem{BF2} Bui, H; Florea, A. {\emph Hybrid Euler-Hadamard product for quadratic Dirichlet $L$-functions in function fields}, {\it Proc. London Math. Soc.}, 2018, no. 1, 65--99.

 \bibitem{DFL} David, C; Florea, A; Lalin, M. {\emph Mean values of cubic $L$-functions over function fields},, {\it Algebra and Number Theory},  {\bf 16} (2022), no. 5, 1259--1326.

\bibitem{DDLL} P. Darbar, P.; David, C; Lalin, M; Lumley, L. {\emph Asymmetric distribution of extreme values of cubic
$L$-functions at $s = 1$}, {\it J. Lond. Math. Soc.}, {\bf 110} (2024), e12996.  

\bibitem{DL} Darbar, P.; Lumley, A. {\emph Selberg’s central limit theorem for quadratic Dirichlet $L$-functions over function fields}, {\it Monatsh Math}, {\bf 201} (2023), 1027--1058.

\bibitem{DM} Darbar, P.; Maiti, G.
{\emph Large values of quadratic Dirichlet $L$-functions}, {\it Math. Ann.}, {\bf 392} (2025), 4573--4605.

\bibitem{DOKIC}
Doki\'{c}, D.; Lelas, N.; Vre\'{c}ica, I. {\emph Large values of Dirichlet $L$-functions over function fields}, {\it Int. J. Number Theory}, {\bf 16} (2020), 1081--1109.


\bibitem{FR} D. Faifman, D; Rudnick, Z. {\emph Statistics of the zeros of zeta functions in families of hyperelliptic curves over a finite field}, {\it Compos. Math.}, {\bf 146} (2010), no. 1, 81--101.

 
 \bibitem{Floreanegative} Florea, A. {\emph 
 	Negative moments of $L$-functions with small shifts over function fields}, {\it Int. Math. Res. Not. IMRN}, 2024, no. 3, 2298--2337.
 
 \bibitem{Florea} Florea, A. {\emph The fourth moment of quadratic Dirichlet $L$-functions over function fields}, {\it Geom. Funct. Anal.}, {\bf 27} (2017), no. 3, 541--595.
 
 \bibitem{GS} Granville, A.; Soundararajan, K. {\emph The distribution of values of $L(1, \chi_d)$}, {\it Geometric and Funct. Anal},
 {\bf 13} (2003), 992--1028.
 
 \bibitem{GSzeta} Granville, A.; Soundararajan, K. {\emph Extreme values of  $|\zeta(1+it)|$}, 
 {\it Ramanujan Math. Soc. Lect. Notes Ser.}, {\bf 2}
 Ramanujan Mathematical Society, Mysore, 2006, 65--80.
 
 \bibitem{Jung}
  Jung, H.  {\emph Note on the mean value of $L(\frac{1}{2}, \chi_D)$ in the hyperelliptic ensemble}, {\it J. Number Theory}, {\bf 133} (2013), 2706--2714.

\bibitem{Lamzouri} Lamzouri, Y. {\emph Distribution of values of  $L$-functions at the edge of the critical strip}, {\it Proc. Lond. Math. Soc.}, {\bf 100} (2010), 835--863.
 

\bibitem{Lumley2}
Lumley, A. {\emph Moments and distribution of values of $L$-functions over function fields inside the critical strip}, {\it Acta Arith.}, {\bf 201} (2021), 329--369.


\bibitem{Lumley}
Lumley, A. {\emph Complex moments and the distribution of values of $L(1, \chi_D)$ over function fields with applications to class numbers}, {\it Mathematika}, {\bf 65 } (2019), 236--271.

\bibitem{MV} Montgomery, H. L; Vaughan, R. C. {\emph Extreme values of Dirichlet $L$-functions at $1$}, {\it Number theory
in progress}, de Gruyter, Berlin, 1999, 1039--1052.



 
 \bibitem{ROS}
  Rosen, M. {\emph Number theory in function fields}, {\it Graduate Texts in Mathematics}, {\bf 210} (2002), Springer-Verlag, New York.
  
  \bibitem{Rud}  Rudnick, Z. {\emph Traces of high powers of the Frobenius class in the hyperelliptic ensemble}, {\it Acta Arith}, {\bf 143} (2010), 81--99.


\bibitem{WEIL}  Weil, A. {\emph Sur les courbes alg\'{e}briques et les vari\'{e}t\'{e}s qui s'en déduisent}, {\it Actualit\'{e}s Sci. Ind.}, {\bf  1041} (1948), Hermann et Cie., Paris.
	
	\end{thebibliography}
  \end{document}